\documentclass[12pt]{amsart}
\usepackage{amscd,amssymb,graphics,lineno}
\usepackage{mathrsfs} 

\newtheorem{theorem}{Theorem}[section]
\newtheorem{corollary}[theorem]{Corollary} 
\newtheorem{lemma}[theorem]{Lemma}
\newtheorem{proposition}[theorem]{Proposition}
\theoremstyle{definition}

\theoremstyle{remark}
\newtheorem{remark}[theorem]{Remark}

\newtheorem{example}[theorem]{Example}
\numberwithin{equation}{section}

\newcommand{\abs}[1]{\lvert#1\rvert}
\def\norm#1{\left\Vert#1\right\Vert}
\def\Q {{\mathbb Q}}
\def\s{{\mathbb S}}
\def\I {{\mathbb I}}
\def\C {{\mathbb C}}
\def\T {{\mathbb T}}
\def\N{{\mathbb N}}
\def\Z {{\mathbb Z}}

\def\R{{\mathbb R}}

\def\e{\varepsilon}
\def\Homeo{{\mbox{\rm Homeo}\,}}
\def\Aut{{\mbox{\rm Aut}\,}}
\def\RUCB{{\mbox{\rm RUCB}\,}}
\def\UCB{{\mbox{\rm UCB}\,}}
\def\LUCB{{\mbox{\rm LUCB}\,}}

\def\H {{\mathcal H}}
\newcommand\B{{\mathscr B}}

\newcounter{quest}
\stepcounter{quest}

\begin{document}

\title[Amenability versus property $(T)$]
{Amenability versus property $(T)$ for non locally compact topological groups}

\author[V.G. Pestov]{Vladimir G. Pestov}
\thanks{Special Visiting Researcher of the program Science Without Borders of CAPES (Brazil), processo 085/2012.}

\address{Department of Mathematics and Statistics, 
University of Ottawa, 585 King Edward Ave., Ottawa, Ontario, Canada K1N 6N5}

\address{Departamento de Matem\'atica, Universidade Federal de Santa Catarina, Trindade, Florian\'opolis, SC, 88.040-900, Brazil}
\email{vpest283@uottawa.ca}




\begin{abstract} 
        For locally compact groups amenability and Kazhdan's property (T) are mutually exclusive in the sense that a group having both properties is compact. This is no longer true for more general Polish groups. However, a weaker result still holds for SIN groups (topological groups admitting a basis of conjugation-invariant neighbourhoods of identity): if such a group admits sufficiently many unitary representations, then it is precompact as soon as it is amenable and has the strong property $(T)$ (i.e. admits a finite Kazhdan set). If an amenable topological group with property $(T)$ admits a faithful uniformly continuous representation, then it is maximally almost periodic. 
        In particular, an extremely amenable SIN group never has strong property $(T)$, and an extremely amenable subgroup of unitary operators in the uniform topology is never a Kazhdan group.
        This leads to first examples distinguishing between property $(T)$ and property $(FH)$ in the class of Polish groups. 
Disproving a 2003 conjecture by Bekka, we construct a complete, separable, minimally almost periodic topological group with property $(T)$, having no finite Kazhdan set.
Finally, as a curiosity, we observe that the class of topological groups with property $(T)$ is closed under arbitrary infinite products with the usual product topology.
\end{abstract}

\maketitle

\section{Introduction}

One of the most immediate observations about groups with property $(T)$ is the following (cf. e.g. Thm. 1.1.6 in \cite{BdlHV}).

\begin{theorem}
A locally compact group $G$ with property $(T)$ is amenable if and only if $G$ is compact.
\end{theorem}

This is a direct consequence of Reiter's condition $(P_2)$ which states that a locally compact group $G$ is amenable if and only if the unitary $G$-module $L^2(G)$ has almost invariant vectors, and hence by the definition of property $(T)$ there must be an invariant vector of norm one. Clearly, the argument does not extend to non locally compact groups, because typically such groups admit even no quasi-invariant measure. A result of Andr\'e Weil (cf. Appendix B in \cite{GTW} for a discussion and references) says that a Polish group equipped with a left quasi-invariant sigma-finite Borel measure is locally compact.

Not only the proof, but the result itself is in general invalid for Polish non-locally compact groups. 
The group $\Homeo_+(\I)$ of orientation-preserving homeomorphisms of the unit interval with the compact-open topology is amenable, in fact
even {\em extremely amenable} \cite{P1}, that is, has a fixed point in every compact space upon which it acts continuously, yet admits no non-trivial unitary representations \cite{megrelishvili}, and so has property $(T)$ for trivial reasons. 

The situation does not change if we limit ourselves to topological groups admitting sufficiently many strongly continuous unitary representations. The unitary group $U(\ell^2)$ with the strong operator topology is a Polish group having property $(T)$, in fact even strong property $(T)$, meaning it contains a finite, and not just compact, Kazhdan set (Bekka \cite{Bek3}). At the same time, the group $U(\ell^2)_s$ is amenable, that is, every continuous action of $U(\ell^2)_s$ on a compact space admits an invariant regular Borel probability measure (de la Harpe \cite{dlH}). Moreover, according to Gromov and Milman \cite{GrM}, this group is extremely amenable. 

It turns out that amenability and strong property $(T)$ are still incompatible in a large class of non locally compact groups. 
A topological group $G$ is called a {\em SIN group} (from Small Invariant Neighbourhoods) if neighbourhoods of the identity invariant under conjugation form a neighbourhood basis, or, equivalently, if the left and the right uniform structures on $G$ coincide. 

Using Bekka's concept of an amenable unitary representation \cite{Bek1} and a result from \cite{GP}, we prove the following.  

\begin{theorem}
        Let $G$ be a SIN group whose continuous unitary representations separate points and closed sets. If $G$ is amenable and has strong property $(T)$, then it is precompact.
\label{th:main}
\end{theorem}

A topological group $G$ is {\em minimally almost periodic} if it has no non-trivial continuous finite-dimensional unitary representations.

\begin{corollary}
Let $G$ be a minimally almost periodic SIN group admitting a non-trivial unitary representation. If $G$ is amenable, then it does not have strong property $(T)$.
\end{corollary}

This result leads us to conclude that a number of well-known infinite dimensional Polish groups do not have strong property $(T)$. Among them are:
\begin{itemize}
\item
        The group $L^0((0,1),K)$ of all measurable maps from the unit interval to a compact group $K$ with the topology of convergence in measure;
\item
The full group of the tail equivalence relation with the uniform topology;
\item
The unitary group $U(R)$ of the hyperfinite factor of type $II_1$ with the strong (equivalently: Hilbert-Schmidt) topology.
\end{itemize}

The last example is especially interesting in view of Bekka's theorem quoted above. We do not know whether all of the groups listed above do not have property $(T)$ either. In a separate section, we establish this for the former group in the special case where $K=\T$, the circle rotation group.

Let us recall that a topological group is {\em maximally almost periodic} if finite-dimensional continuous unitary representations separate points.
Every topological subgroup of the unitary group $U(\H)$ with the uniform topology is SIN. For such groups we have:

\begin{theorem}
        Let $G$ be a topological subgroup of the unitary group $U(\H)$ with the uniform operator topology. If $G$ is amenable and has property $(T)$, then $G$ is maximally almost periodic.
\end{theorem}

\begin{corollary}
  Let $G$ be a minimally almost periodic topological subgroup of the unitary group $U(\H)$ with the uniform operator topology. If $G$ is amenable, then it does not have property $(T)$.
\end{corollary}

For example, the group $U_C(\ell^2)$ of unitary operators of the form ${\mathbb{I}}+K$, where $K$ is compact, with the uniform operator topology does not have property $(T)$.

Two of the above examples allow, for the first time, to distinguish between the property $(T)$ and the property $(FH)$ in the class of Polish groups. 

A topological group $G$ has property $(FH)$ if every continuous action of $G$ by affine isometries of a Hilbert space has a fixed point. It is always true that $(T)\Rightarrow (FH)$, and the converse implication holds for sigma-compact locally compact groups (the Delorme-Guichardet theorem, cf. \cite{BdlHV}, Th. 2.12.14). For more general topological groups the two properties have been distinguished: as observed e.g. in \cite{dC}, every (necessarily uncountable) discrete group having Bergman's property \cite{bergman} has property $(FH)$, while it is well-known and easily proved that a discrete group with property $(T)$ is finitely generated (cf. \cite{BdlHV}, Th. 1.3.1). An even stronger property than $(FH)$ is the property $(OB)$ studied by Rosendal \cite{rosendal}.
A topological group has property $(OB)$ if every continuous left-invariant pseudometric on $G$ is bounded. Equivalently, every orbit of each continuous action of $G$ by isometries on a metric space is bounded. 
The Polish groups $L^0((0,1),\T)$ and $U_C(\ell^2)$ do not have property $(T)$, yet both have property $(OB)$. 

Next we disprove a 2003 conjecture of Bachir Bekka \cite{Bek3}. The pre-eminent example of a Kazhdan group which is not strongly Kazhdan is the circle rotation group $\T=U(1)$, easily shown to have no finite Kazhdan subset. More generally, the same holds for every compact group which is amenable in the discrete topology (Bekka, {\em ibid.}) 
On the other hand, known examples of minimally almost periodic locally compact groups with property $(T)$ such as $SL_3(\R)$ are strongly Kazhdan. Bekka has proved that this is no coincidence: a locally compact Kazhdan group $G$ is strongly Kazhdan if and only if the Bohr compactification of $G$ is strongly Kazhdan. In particular, if a locally compact group $G$ is minimally almost periodic (so its Bohr compactification $bG$ is trivial) and Kazhdan, then it must be strongly Kazhdan.
Bekka has conjectured that the criterion remains true for arbitrary topological groups ({\em ibid.}, p. 512). 

However, we show a rather natural example of a complete (non-metrizable) topological group $G$ which has property $(T)$, is minimally almost periodic, yet does not have the strong property $(T)$. 
This group is obtained starting from any infinite simple Kazhdan group $\Gamma$ as an intermediate subgroup between the direct sum and direct product of countably infinitely many copies of $\Gamma$, equipped with the chain topology with regard to a sequence of compact sets of the form $B_n^{\N}$ (powers of the balls in the word metric on $\Gamma$). 

It is well-known that the product of finitely many topological groups with Kazhdan's property $(T)$ is again Kazhdan.
As our concluding result, we show that the same holds for arbitrary infinite products of Kazhdan groups equipped with the usual product topology.

\section{Some properties of topological groups}

Let $G$ be a (Hausdorff) topological group.

\subsection{Uniformities}
The {\em left uniformity} on $G$ is determined by the entourages of the diagonal
\[V_L=\{(x,y)\in G\times G\colon x^{-1}y\in V\},\]
where $V$ runs over a neighbourhood basis of the identity. The {\em right uniformity,} similarly, is given by the entourages of the form
\[V_R=\{(x,y)\in G\times G\colon xy^{-1}\in V\}.\]
The {\em two-sided} uniformity on $G$ is the supremum of the left and the right uniformities, and its basis of entourages of the diagonal consists of sets
\[V_L\cap V_R = \{(x,y)\in G\times G\colon x^{-1}y\in V\mbox{ and }xy^{-1}\in V\}.\]
All three uniformities are compatible with the topology of the group.

\subsection{Groups with Small Invariant Neighbourhoods}
A topological group $G$ is called a {\em group with small invariant neighbourhoods,} or a {\em SIN group}, if the left and right uniform structures of $G$ coincide. This is the case if and only if conjugation-invariant neighbourhoods form a basis at identity of $G$.
Equivalently, SIN groups are those whose topology is determined by the family of all continuous bi-invariant pseudometrics on $G$. In particular, is $G$ is a metrizable SIN group, there is always a compatible bi-invariant metric on $G$.

Every abelian or compact group has small invariant neighbourhoods, but there are numerous examples of infinite-dimensional SIN groups of importance outside of these two classes.

\subsection{Completeness}
The group $G$ is {\em complete} if it is complete with regard to the two-sided uniformity. An equivalent condition is that $G$ is closed in every topological group containing $G$ as a topological subgroup. Every topological group $G$ embeds as a dense topological subgroup in a unique complete group, $\hat G$, called the {\em completion} of $G$.

For more on uniform structures and completeness, see \cite{RD}.

\subsection{Polish groups}
A topological group $G$ is {\em Polish} if it is complete, metrizable and separable. The class of Polish groups is, in a sense, the second most natural class of topological groups beyond the locally compact case, see e.g. \cite{BK}. At the same time, not all concrete infinite-dimensional groups of importance are necessarily Polish.

\subsection{Amenability}
Denote by $\RUCB(G)$ the collection of all complex-valued bounded right uniformly continuous functions on $G$. A 
topological group $G$ is {\em amenable} if it satisfies one of the following equivalent conditions.

\begin{itemize}
\item There is a left-invariant mean on the space $\RUCB(G)$.
\item Every continuous action of $G$ on a compact space $X$ admits an invariant Borel probability measure on $X$.
\item Every affine continuous action of $G$ on a compact convex set has a fixed point.
\end{itemize}

The following in a general context was observed by de la Harpe \cite{dlH}, see also \cite{paterson}.

\begin{theorem}
Amenability is preserved under passing to (i) the union of an increasing chain of topological subgroups; (ii) the completion; (iii) an everywhere dense subgroup.
\label{th:union}
\end{theorem}

\subsection{Extreme amenability}
The following notion is a peculiarity of the non locally compact case. A topological group $G$ is {\em extremely amenable} if every continuous action of $G$ on a compact space admits a fixed point. The book \cite{P06} is a reference to this concept. Note that every extremely amenable topological group is minimally almost periodic, that is, admits no continuous finite dimensional unitary representations (because for every such representation $\pi\colon G\to U(n)$ the composition of $\pi$ and the action by left translations would give an action of $G$ on the compact space $U(n)$ without fixed points). 

\subsection{Property $(T)$}
The notion of a Kazhdan group makes perfect sense beyond the locally compact case. A unitary representation $\pi$ of a topological group $G$ in a Hilbert space $\H$ {\em has almost invariant vectors} if for every $\e>0$ and every compact $K\subseteq G$ there is a $\xi\in\H$ with $\norm\xi=1$ and $\norm{\xi-\pi_g\xi}<\e$ for all $g\in K$. A topological group $G$ has property $(T)$ if every strongly continuous unitary representation of $G$ having almost invariant vectors has an invariant vector (of norm one). In this case, there is necessarily a compact {\em Kazhdan set,} $Q\subseteq G$, and a {\em Kazhdan constant,} $\e>0$, so that if there is a $(Q,\e)$-almost invariant vector (in an obvious sense), $\xi$, of norm one, then there is an invariant vector of norm one. If a Kazhdan group admits a {\em finite} Kazhdan set $Q$, then it is called {\em strongly Kazhdan} (or: having {\em strong property $(T)$}).

The book \cite{BdlHV} is an encyclopedic reference to the subject.

\subsection{Property $(FH)$}
A topological group $G$ has {\em property $(FH)$} if every continuous action of $G$ by affine isometries on a Hilbert space has a fixed point. An equivalent condition is: orbits of every continuous action of $G$ by affine isometries on a Hilbert space are bounded. (Cf. ``Lemma of the centre'' in \cite{BdlHV}, Section 2.2.)

For second countable locally compact groups $G$ the properties $(T)$ and $(FH)$ are equivalent, which result is known as the {\em Delorme--Guichardet theorem,} see Ch. 2 in \cite{BdlHV}.
While the implication $(T)\Rightarrow (FH)$ (due to Delorme \cite{delorme}) holds for every topological group, the converse implication $(FH)\Rightarrow (T)$ (established by Guichardet \cite{guichardet} for sigma-compact locally compact groups) is in general invalid already for uncountable discrete groups. 

A (discrete) group $G$ has {\em Bergman's property} \cite{bergman} if, whenever $G$ is represented as a union of an increasing countable chain $(W_i)_{i=1}^\infty$ of subsets, a suitable finite power of some $W_i$ equals $G$. (For instance, the group $S_\infty$ viewed as a discrete group has Bergman property \cite{bergman}, and the collection of known such groups is growing fast, including, in particular, the unitary group $U(\ell^2)$ \cite{RR}, and many others.) 
An equivalent reformulation of Bergman property is: the orbits of every action of $G$ by isometries on a metric space are bounded. 

As a consequence, noted in particular in \cite{dC}, every uncountable discrete group $G$ with Bergman property has property $(FH)$ but not property $(T)$, because a discrete group with property $(T)$ is finitely generated (cf. \cite{BdlHV}, Th. 1.3.1.)

\subsection{Property $(OB)$}
The following natural extension of Bergman property to topological groups was studied by Rosendal \cite{rosendal}. Say
that a topological group $G$ has {\em property $(OB)$} if for every continuous action of $G$ on a metric space by isometries all orbits are bounded. If a group has the Bergman property as a discrete group (e.g. $S_\infty$ and $U(\ell^2)$), it has the property $(OB)$. A topological group $G$ is {\em bounded} in the sense of Hejcman \cite{hejcman} and Atkin \cite{atkin} if for every neighbourhood of identity $V$ in $G$ there are a finite subset $F\subset G$ and a natural number $n$ with $FV^n=G$. An example of a bounded group is $U(\ell^2)_u$ \cite{atkin}. Every bounded group has property $(OB)$. 

The property $(OB)$ clearly implies property $(FH)$. The converse in general fails, e.g. consider any infinite discrete Kazhdan group such as $SL_3(\Z)$.

For many Polish groups with property $(OB)$ it remains unknown whether or not they have property $(T)$. This is the case e.g. for the unitary group $U(\ell^2)$ with the uniform operator topology, the group $\Aut(X,\mu)$ of measure-preserving transformations of a standard Lebesgue measure space with the weak topology, etc. 

In this paper we will distinguish between properties $(FH)$ and $(T)$ in the class of Polish groups, by proving that two among the most common examples of Polish groups with property $(OB)$ do not have property $(T)$. 

Here is a useful simple lemma to help establish property $(OB)$ of some groups.
 
 \begin{lemma}
 Let a topological group $G$ contain an increasing chain of subgroups $K_n$ whose union is dense. Suppose that $G$ is metrizable with a left-invariant metric $d$ and that for every $\e>0$ there exists a finite $N$ with the property that the $N$-th power of the $\e$-neighbourhood of identity in $K_n$ formed with regard to $d$ is all of $K_n$. Then $G$ is bounded and in particular has property $(OB)$. 
 \label{l:ob}
 \end{lemma}
 
 \begin{proof}
 Denote by $V_\e$ the $\e$-ball around identity in $G$. By lemma's assumptions, for some $N$ the $N$-th power of $V_\e$ will contain the union of $K_n$, $n\in\N$, and thus the $(N+1)$-th power of $V_\e$ will contain the closure of $\cup_{i=1}^\infty K_n$, that is, $G$. 
 \end{proof}

\section{Some examples of topological groups}
\label{s:examples}

\subsection{$\Homeo_+(\I)$\label{homeo}} The group $\Homeo_+(\I)$ consists of all orientation-preserving (or: endpoint-preserving) homeomorphisms of the closed unit interval $\I=[0,1]$ and is equipped with the compact-open topology. This group is isomorphic, as a topological group, to the group $\Homeo_+(\R)$ with the compact-open topology. The group is Polish, not SIN, and extremely amenable \cite{P1}, in particular amenable. As proved by Megrelishvili \cite{megrelishvili}, the Polish group $\Homeo_+(\R)$ admits no non-trivial strongly continuous representations by isometries in reflexive Banach spaces, in particular, no nontrivial unitary representations. For this reason, it has property $(T)$ in a trivial way.
The group $\Homeo_+(\I)$ is bounded (\cite{atkin}, Remark 6.8(b)). 

\subsection{$U(\H)_s$} The group $U(\H)$ of unitary operators on a Hilbert space $\H$, equipped with the strong operator topology, will be denoted $U(\H)_s$. This group is complete and amenable for every $\H$ (for $\H$ infinite-dimensional it follows from Th. \ref{th:union} because $U(\H)$ is then approximated by unitary subgroups of finite rank, cf. \cite{dlH}). If $\dim\H$ is infinite, then $U(\H)_s$ is extremely amenable \cite{GrM}.
The group is compact if and only if $\dim\H<\infty$, and the same is true of the SIN property. The group $U(\H)_s$ has property $(OB)$, which can be easily verified directly. 
If $\H=\ell^2$ is separable infinite-dimensional, then $U(\ell^2)$ has property $(T)$ (Bekka \cite{Bek3}). For $\H$ finite-dimensional, property $(T)$ follows from compactness.

\subsection{$U(\H)_u$} The uniform operator topology on the same group $U(\H)$ will be marked with a subscript $u$. The group $U(\H)_u$ is complete and SIN. If  $\dim\H=\infty$, the group $U(\H)_u$ is non-separable, nor is it amenable \cite{dlH}. The group $U(\H)_u$ is bounded (\cite{atkin}, (3.5)), so has property $(OB)$.
It remains unknown if $U(\ell^2)_u$ has property $(T)$ \cite{Bek3}.

\subsection{$U_C(\H)$\label{uch}}
The Fredholm unitary group $U_C(\H)$ of all compact perturbations of the identity, that is, of all unitary operators $\I+K$, where $K$ is a compact operator in a Hilbert space $\H$, equipped with the uniform operator topology, is a Polish SIN group.  
This group is extremely amenable (\cite{GrM}, see also \cite{P06}, Corol. 4.1.17). It is bounded, so has property $(OB)$ (cf. Remark 6.9 in \cite{atkin_hokkaido}, and Proposition 6.5 and Remark 6.8 in \cite{atkin}). In particular, the Polish group $U_C(\H)$ has property $(FH)$, while we will show that it does not have property $(T)$.

\subsection{$U(\infty)_2$\label{uinfty2}} The group $U(\infty)_2$ consists of all unitary operators on $\ell^2$ of the form $\I+T$, where $T$ is of Schatten class $2$. Equipped with the (non-normalized) Hilbert-Schmidt metric, this is a Polish SIN group. It is extremely amenable (Gromov and Milman \cite{GrM}, see also \cite{P06}, Corol. 4.1.13), and SIN. As will be shown, this group does not have property $(T)$. The Hilbert-Schmidt metric on this group is bi-invariant and unbounded, and so the group does not have property $(OB)$. We do not know if this group has property $(FH)$.

\subsection{$U(M)_s$\label{ss:ums}} Let $M$ be a von Neumann algebra. The unitary group $U(M)$ of $M$ will be equipped with the strong operator topology, which coincides with the weak$^\ast$ topology $\sigma(M,M_\ast)$ with regard to the predual $M_\ast$. This topology is always complete, and is Polish if and only if $M$ is separable. 
The group $U(M)_s$ is amenable if and only if $M$ is injective \cite{dlH79,haagerup}, in which case $U(M)_s$ decomposes into a direct product of an extremely amenable group and a compact group \cite{GP}. 
In some cases, $U(M)_s$ has property $(T)$ (e.g. where $M=B(\ell^2)$ and $U(M)_s=U(\ell^2)_s$, Bekka's theorem), in some others, does not (if $M=L^\infty(0,1)$, we will show that the corresponding unitary group $U(M)_s=L^0(X,\mu;\T)$, cf. \ref{loxg} below, does not have property $(T)$).
It is unclear whether property $(T)$ of the topological group $U(M)_s$ and the property $(T)$ of the von Neumann algebra $M$ as defined in \cite{connes_jones} are linked in any way.

\subsection{$U(M)_2$\label{ur2}} Suppose that $M$ is a von Neumann factor of type $II_1$. Then the strong operator topology on $U(M)$ coincides with the topology given by the Hilbert-Schmidt norm 
\[\norm x_2 = \tau(x^\ast x)^{1/2},\]
where $\tau$ is the unique normalized trace on $M$ (cf. e.g. Prop. 9.1.1 in \cite{jones}). We will denote the unitary group by $U(M)_2$. This group is SIN, in fact the Hilbert-Schmidt distance is bi-invariant. (As pointed out to me by Philip Dowerk and Andreas Thom, and independently by Sven Raum, a von Neumann algebra $M$ is finite if and only if the group $U(M)_s$ is SIN.) 

In the separable case, amenability of $U(M)_2$ means $M=R$ is the hyperfinite factor of type $II_1$, in which case the group $U(R)_2$ is in fact extremely amenable \cite{GP}. We will show that the group $U(R)_2$ does not have strong property $(T)$, however we were unable to verify whether it has property $(T)$. It follows from Lemma \ref{l:ob} that $U(R)_2$ is a bounded group: a finite power of each neighbourhood of identity is the entire group.

The tracial ultrapower $R^\omega$ of the hyperfinite factor $R$ over a non-principal ultrafilter on the integers has attracted a growing attention recently in view of Connes' Embedding Conjecture \cite{connes-injective,CL}. We do not know whether the unitary group $U(R^\omega)_2$ is Kazhdan or even strongly Kazhdan. The boundedness property of $U(R)_2$ as stated above is preserved by metric ultraproducts, and consequently $U(R^\omega)_2$ is a bounded group. 

\subsection{$L^0(X,\mu;G)$\label{loxg}} Let $(X,\mu)$ be a standard Borel space equipped with a probability measure, and let $G$ be a locally compact Polish group. Denote by $L^0(X,\mu; G)$ the collection of all (equivalence classes of) Borel measurable maps from $X$ to $G$, equipped with the pointwise group operations and the topology of convergence in measure. The group $L^0(X,\mu; G)$ is Polish. It is SIN if and only if $G$ is SIN, in particular if $G$ is abelian or compact. 

If the locally compact group $G$ is amenable and the measure $\mu$ is diffuse (non-atomic), then $L^0(X,\mu;G)$ is extremely amenable (\cite{P02}; \cite{P06}, p. 158). (For compact $G$ it was proved by Glasner \cite{Gl1} and, independently, by Furstenberg--Weiss, unpublished. The problematics was initiated by Herer--Christensen \cite{HC} for $G=U(1)$ and the so-called pathological diffuse submeasures. For further developments in the case of submeasures, see \cite{FS,sabok}.) It follows that for every measure $\mu$ and an amenable locally compact $G$, the group $L^0(X,\mu;G)$ is amenable. The case where $G=U(1)$ is the circle rotation group overlaps with Example \ref{ss:ums}, because the group $L^0(X,\mu;U(1))$ is isomorphic to the unitary group $U(L^\infty(0,1))_s$. We will show that this group does not have property $(T)$, while it is easy to see that it has property $(OB)$, and consequently property $(FH)$.

\subsection{$\Aut^\ast(X,\mu)_u$} The group of measure class preserving automorphisms of a standard Borel space with probability measure $\mu$ is equipped with the uniform distance
\begin{equation}
d_u(\sigma,\tau)=\mu\{x\in X\colon \sigma(x)\neq\tau(x)\}.
\label{eq:uniform}
\end{equation}
The group is complete. If the measure $\mu$ is diffuse, the group is non-separable. In the case of diffuse measure,
it is also unknown whether the group $\Aut^\ast(X,\mu)_u$ is either amenable or even extremely amenable (asked in \cite{GP}), or whether it has property $(T)$.
The group is not SIN. It has the Bergman property \cite{DHU}, therefore the property $(OB)$.

\subsection{$\Aut^\ast(X,\mu)_w$} The weak topology on the group $\Aut^\ast(X,\mu)$ is the strong operator topology determined by the quasi-regular representation $\gamma$ of the group in the space $L^p(X,\mu)$, $1\leq p<\infty$, defined 
for every $\tau\in\Aut^\ast (X,\mu)$ and $f\in L^p(X)$ by
\[\gamma_\tau(f)(x)=f(\tau^{-1}x)
\root p\of{\frac{d \mu\circ\tau^{-1}}{d\mu}(x)},\]
where $\mu\circ\tau^{-1}=\tau_\ast(\mu)$ is the pushforward measure, and
$d/d\mu$ the Radon--Nikod\'ym derivative. We mark the weak topology with the subscript $w$. The group $\Aut^\ast(X,\mu)_w$ is Polish. If $\mu$ is diffuse, then $\Aut^\ast(X,\mu)_w$ is extremely amenable \cite{GP}. From here it is easy to deduce, with a view of Example \ref{ss:sinfty}, that for every measure $\mu$ the group $\Aut^\ast(X,\mu)_w$ is amenable. 
In the case of a diffuse measure $\mu$, it is unknown whether $\Aut^\ast(X,\mu)_w$ has property $(T)$. 

\subsection{$S_\infty$\label{ss:sinfty}} In a case where the measure space $(X,\mu)$ is purely atomic, the uniform and the weak topology on the group $\Aut^\ast(X,\mu)$ coincide. If in addition $X$ has infinitely many atoms, the topological group $\Aut^\ast(X,\mu)$ is isomorphic to the full symmetric group $S_\infty$ of all bijections of a countably infinite set $\omega$, equipped with the topology of simple convergence on the discrete space $\omega$. The group $S_\infty$ is Polish, not SIN, and amenable (as it can be approximated by a chain of finite permutation groups). This group is not extremely amenable (as first observed in \cite{P1}), and some of its topological dynamical properties are now well understood \cite{Gl-W1}, see also \cite{P06}, Sect. 6.3. The same proof as in Bekka's article \cite{Bek3}, using the result from \cite{lieberman}, shows that $S_\infty$ (as a Polish group) has property $(T)$. This can be now deduced from a very general result by Tsankov \cite{tsankov}.

\subsection{$\Aut(\Q,\leq)$} The group of all order-preserving bijections of the rationals with the topology induced from $S_{\infty}$ is a Polish group. It is extremely amenable \cite{P1} and has strong property $(T)$ \cite{tsankov}. Just like $S_{\infty}$, the group $\Aut(\Q,\leq)$ has the Bergman property as a discrete group \cite{DH} and consequently the property $(OB)$.

\subsection{$\Aut(X,\mu)_u$\label{ss:mp}}
The group $\Aut(X,\mu)$ of measure-preserving transformations of a measure space $(X,\mu)$ equipped with the uniform metric as in Eq. (\ref{eq:uniform})
and the corresponding topology is a complete SIN group. If $\mu$ is diffuse, the group is non-separable, and it is non-amenable \cite{GP}. This topological group clearly has property $(OB)$, in fact even the Bergman property \cite{DHU}. We do not know whether it has property $(T)$, or even strong property $(T)$.

\subsection{$\Aut(X,\mu)_w$\label{ss:mpw}}
The same group equipped with the weak topology is extremely amenable \cite{GP}. We do not know if it is Kazhdan or strongly Kazhdan.

\subsection{$[{\mathscr R}]$\label{full}}
Let $\mathscr R$ be a measure class preserving ergodic countable Borel equivalence relation on a standard Borel measure space $X$ equipped with a diffuse probability measure $\mu$. Denote by $[{\mathscr R}]$ the full group of $\mathscr R$ in the sense of Dye \cite{dye}, that is, a subgroup of $\Aut^\ast(X,\mu)$ formed by all transformations $\tau$ with the property $(x,\tau(x))\in {\mathscr R}$ for a.e. $x\in X$.
When equipped with the uniform topology as a subgroup of $\Aut^\ast(X,\mu)$, the group $[{\mathscr R}]$ is Polish. The group $[{\mathscr R}]$ is amenable if and only if it is extremely amenable if and only if the relation $\mathscr R$ is hyperfinite \cite{GP}. If the relation $\mathscr R$ is measure-preserving, the 
full group $[{\mathscr R}]$ is SIN. If $\mathscr R$ is hyperfinite, the full group has the Bergman property \cite{DHU}.
We will show that the full group of a hyperfinite measure-preserving equivalence relation does not have strong property $(T)$, and we do not know whether it has property $(T)$.

\section{Amenable representations}

A unitary representation $\pi$ of a group $G$ in a Hilbert space $\H$ is {\em amenable} (in the sense of Bekka \cite{Bek1}) if there exists
a state $\phi$ on the von Neumann algebra ${\mathcal B}(\H)$ 
invariant under the action of $G$ by conjugations: for all $T\in {\mathcal B}(\H)$ and all $g\in G$, one has
\[\phi(T) = \phi(\pi^\ast_g T\pi_g).\]
An useful equivalent reformulation in the language of classical invariant means on function spaces can be found in \cite{P00}, Theorem 7.6. Denote by $\s_\pi$ the unit sphere in the Hilbert space of representation, $\H_{\pi}$, and by $\UCB(\s_{\pi})$ the Banach space of all bounded uniformly continuous functions on the sphere made into a $G$-module in an obvious way. Then $\pi$ is amenable if and only if there is a $G$-invariant mean on $\UCB(\s_{\pi})$. 

Every strongly continuous unitary representation of an amenable locally compact group is amenable \cite{Bek1}. This is no longer true for non-locally compact topological groups. 

\begin{example} The standard representation $\pi$ of the amenable group $U(\ell^2)_s$ in $\ell^2$ is non-amenable, for the following reason. If we identify the standard orthonormal basis in $\ell^2$ with the underlying set of the group $F_2$, the left regular representaion of the latter group, $\lambda_{F_2}$, determines an embedding $F_2\hookrightarrow U(\ell^2)$.
For every $f\in L^\infty(F_2)$ denote by $ m_f$ the corresponding multiplication operator on $\ell^2=\ell^2(F_2)$. The correspondence $f\mapsto m_f$ is a representation of 
the von Neumann algebra $L^\infty(F_2)$ in $\B(\ell^2)$. If we assume that $\pi$ is amenable, with an invariant state $\phi$, then the rule $\psi(f)= \phi( m_f)$ determines an invariant mean on $F_2$, which is a contradiction.
(More generally, the left regular representation of a locally compact group $G$ is amenable if and only if $G$ is amenable \cite{Bek1}.)
\end{example}

At the same time, as noticed in \cite{GP}, the result still holds true in the class of SIN groups.

\begin{theorem}[Giordano and Pestov \cite{GP}] Every strongly continuous unitary representation $\pi$ of an amenable SIN topological group $G$ is amenable. 
\label{th:gp}
\end{theorem}

\begin{proof}
Choose an artibrary vector $\xi$ in the unit sphere $\s_\pi$ of the Hilbert space of representation.
For every bounded uniformly continuous function $f\colon \s_{\pi}\to\C$ define a function $\gamma(f)\colon G\to\C$ by
\[\gamma(f)(g) = f(\pi_g(\xi)).\]
The resulting correspondence $\gamma$ is a linear positive operator of norm one commuting with the action of $G$ from the Banach $G$-module $\UCB(\s_{\pi})$ to the Banach $G$-module $\LUCB(G)$ of all left uniformly continuous bounded functions on $G$. Since $G$ is a SIN group, $\LUCB(G)=\RUCB(G)$, and since $G$ is amenable, there exists a left-invariant mean $\phi\colon \RUCB(G)\to\C$. The image $\phi\circ\gamma\colon \UCB(\s_{\pi})\to\C$ of $\phi$ under the dual operator $\gamma^\ast$ is a $G$-invariant mean on the sphere. 
\end{proof}

\begin{example} Since the unitary group $U(\ell^2)_u$ with the uniform operator topology is SIN, it cannot be amenable in view of the above Theorem \ref{th:gp}, since the tautological representation is non-amenable. This is an alternative proof of de la Harpe's result.
\end{example}

Here is another illustration of the above technique, generalizing a theorem of Valette \cite{valette}.

\begin{theorem}
Let $G$ be a topological group with the following properties:
\begin{enumerate}
\item $G$ is SIN, and
\item $G$ admits an increasing chain of compact subgroups whose union is dense.
\end{enumerate}

Then every unitary representation $\rho$ of a (discrete) group $\Gamma$ that factors through a unitary representation of $G$ is amenable. 
\end{theorem}

\begin{proof}
        Let a homomorphism $h\colon\Gamma\to G$ and a strongly continuous unitary representation $\pi$ of $G$ be such that $\rho=\pi\circ h$. The topological group $G$ is amenable and SIN and so $\pi$ is an amenable representation. Clearly, every $G$-invariant mean on ${\mathcal B}(\H)$ is also $\Gamma$-invariant. 
\end{proof} 

\section{Amenability versus strong property $(T)$ for SIN groups}

The following observation goes back to \cite{BV}.

\begin{lemma}
  Let $\pi$ be an amenable strongly continuous unitary representation of a topological group $G$ in a Hilbert space of dimension $\geq 1$. If $G$ has strong property $(T)$, then $\pi$ contains a finite-dimensional subrepresentation.
  \label{l:fd}
\end{lemma}

\begin{proof}
According to Theorem 5.1 in \cite{Bek3}, a representation $\pi$ of a group $G$ is amenable if and only if the representation $\pi\otimes \bar\pi$ of the group $G$ viewed as discrete contains almost invariant vectors. Since this representation is still strongly continuous with regard to the original topology on $G$, and due to the assumed strong property $(T)$ of $G$, there is an invariant vector for $\pi\otimes \bar\pi$. This is equivalent to saying that $\pi$ contains a finite-dimensional subrepresentation, see e.g. Lemma 2 in \cite{BV}, again applied to $G$ viewed as discrete.
\end{proof}

Since property $(T)$ is expressed in the language of strongly continuous unitary representations, it only makes sense to consider for groups whose topology is determined by such representations, or, equivalently, groups on which continuous positive definite functions separate points and closed subsets. Yet another equivalent description of such groups is that they embed, as topological subgroups, into $U(\H)_s$ for a suitable Hilbert space $\H$.  We will express this property by saying that a topological group $G$ admits {\em sufficiently many unitary representations,} or that $G$ admits a {\em topologically faithful unitary representation}. 

For instance, every locally compact group admits a topologically faithful unitary representation (cf. Proposition 2 in \cite{PT}).
With the exception of the group $\Homeo_+(\I)$ (\ref{homeo}), all topological groups listed in Sect. \ref{s:examples} admit topologically faithful unitary representations.

\begin{theorem}
  Let $G$ be a topological group admitting a topologically faithful unitary representation. Suppose further that $G$ has SIN property. If $G$ is amenable and has strong property $(T)$, then $G$ is precompact.
  \label{th:major}
\end{theorem}

\begin{proof}
        Denote by $\pi$ the universal strongly continuous unitary representation of $G$, that is, the direct sum of all non-trivial pairwise inequivalent strongly continuous unitary representations of $G$ in Hilbert spaces of a density character not exceeding that of $G$. Let $\pi_{fin}$ denote the restriction of $\pi$ to the closure of the direct sum of all finite-dimensional representations of $G$, and denote $\pi_{\infty}=\pi_{fin}^{\perp}$.
  By Theorem \ref{th:gp}, $\pi_{\infty}$ is an amenable representation of $G$ and so, by Lemma \ref{l:fd}, it must be trivial, so $\pi=\pi_{fin}$. We conclude that every strongly continuous unitary representation of $G$ is the sum of finite-dimensional representations, and the family of such representations separates points and closed subsets of $G$. It follows that $G$ is precompact.
\end{proof}

\begin{remark} If $G$ is in addition complete (as is the case with most concrete examples), the conclusion of the theorem becomes ``then $G$ is compact.''
\end{remark}

Since extreme amenability implies both amenability and minimal almost periodicity, one deduces:

\begin{corollary}
        Let $G$ be an extremely amenable topological group with SIN property admitting  a topologically faithful unitary representation. If $G$ is non-trivial, then it does not have strong property $(T)$. \qed
\end{corollary}

As an immediate application of this Corollary, we obtain the following.

\begin{example}
  The following Polish groups do not have strong property $(T)$:
  \begin{itemize}
    \item The unitary group $U(R)_2$ of the hyperfinite factor of type $II_1$ equipped with the Hilbert-Schmidt metric (\ref{ur2}).
    \item The full group $[{\mathscr R}]$ of a hyperfinite measure-preserving equivalence relation equipped with the uniform topology (\ref{full}). 
    \item The Fredholm group $U_C(\H)$ of compact perturbations of identity (\ref{uch}). 
  \item The group $L^0(X,\mu;G)$ of measurable maps from the standard Lebesgue measure space to the amenable locally compact group $G$, equipped with the topology of convergence in measure (\ref{loxg}).
  \end{itemize}
\end{example}

In the next two Sections we will show that the two latter groups do not even have property $(T)$. We do not know if it also applies to the two former groups.

\section{Amenability versus property $(T)$ for groups with uniform operator topology}

\begin{lemma}
        Let $\pi_1$ and $\pi_2$ be two uniformly continuous unitary representations of a topological group $G$. The tensor product $\pi_1\otimes\pi_2$ is again uniformly continuous. 
\end{lemma}

\begin{proof}
        It is enough to observe that, given a $C^\ast$-algebra $A$,
        the mapping $x\mapsto x\otimes x$ from $A$ to the spatial tensor product $A\otimes A$ is uniformly continuous on the unitary group $U(A)$, indeed on every norm bounded subset of $A$:
        \[\norm{x\otimes x-y\otimes y} \leq \norm{x\otimes x-x\otimes y} + \norm{x\otimes y -y\otimes y} \leq (\norm{x}+\norm{y})\norm{x-y}.\]
\end{proof}

\begin{lemma}
        Let $G$ be an amenable topological group with property $(T)$. Then every non-trivial uniformly continuous representation $\pi$ of $G$ contains a non-trivial finite-dimensional subrepresentation.
        \label{l:fdr}
\end{lemma}

\begin{proof}
        By passing to the orthogonal complement to the module of all $G$-fixed vectors, we can assume that $\pi$ contains no fixed vectors. 
        The representation $\pi\otimes\bar\pi$ is uniformly continuous. It is equivalent to the representation of $G$ by conjugations on the space ${\mathcal C}_2(\H)$ of Hilbert-Schmidt operators of $\H$, and the corresponding uniform operator distance lifts to a bi-invariant continuous pseudometric, $d$, on $G$. The set $N=\{x\in G\colon d(e,x)=0\}$ is a closed normal subgroup, and $d$ factors through the factor-homomorphism to a bi-invariant continuous metric $\bar d$ on $G/N$. 
        The metric group $(G/N,\bar d)$ is amenable and has the property $(T)$, being a continuous homomorphic image of $G$. 
        Similarly, $\pi$ factors to a unitary representation, $\rho$, of $G/N$, which is amenable by Theorem \ref{th:gp}.  
        
        Let $(Q,\e)$ be a Kazhdan pair for $G/N$, where $Q$ is compact. 
        Select a finite $\e/4$-net, $F$, for the compact set $Q$ with regard to the metric $\bar d$. By Theorem 5.1 in \cite{Bek3}, $\rho\otimes\bar\rho$ possesses a $(F,\e/2)$-almost invariant vector, $\xi$, of norm one. Such a vector is $(Q,\e)$-almost invariant: if $g\in Q$, there is $g_1\in F$ with $\bar d(g,g_1)<\e/4$, and so
        \[\norm{\xi-g\xi}\leq\norm{\xi-g_1\xi}+\norm{g_1\xi-g\xi} <\e/2 + 2\bar d(g_1,g)=\e.\]
        Therefore, there is an invariant vector for $\rho\otimes\bar\rho$ and consequently for $\pi\otimes\bar\pi$. This is equivalent to the existence of an invariant finite-dimensional subspace of $\H$. Since $\pi$ has no invariant vectors, we conclude.
\end{proof}

\begin{theorem}
        Let $G$ be a topological group admitting a faithful uniformly continuous unitary representation. Suppose $G$ is amenable and has property $(T)$. Then $G$ is maximally almost periodic.
\end{theorem}

\begin{proof}
        Let $\pi\colon G\to U(\H)_u$ be an injective uniformly continuous unitary representation. 
        Denote $\pi_{fin}$ the direct sum of all finite-dimensional subrepresentations of $G$. Lemma \ref{l:fdr} implies that the unitary representation $\pi_{fin}^\perp$ is trivial, and so continuous finite-dimensional unitary representations of $G$ separate points.
\end{proof}

\begin{corollary}
        No minimally almost periodic topological group admitting a faithful uniformly continuous representation is at the same time amenable and Kazhdan. \qed
\end{corollary}

\begin{corollary}
        An extremely amenable topological group admitting a faithful uniformly continuous representation is not Kazhdan. \qed
\end{corollary}

\begin{example}
        The Fredholm group $U_C(\H)$ of compact perturbations of identity is extremely amenable, and is a topological subgroup of $U(\ell_2)$ with the uniform operator topology. Consequently, it does not have property $(T)$. 
        
        At the same time, as shown by Atkin \cite{atkin,atkin_hokkaido}, the group $U_C(\H)$ is bounded (cf. Subs. \ref{uch}), and so has property $(OB)$ and property $(FH)$. To our knowledge, this is the first known example of a Polish group with property $(FH)$ which is not a Kazhdan group. Another example, of a different nature, will appear in the following section.
\end{example}

\begin{example}
        The group $U(\infty)_2$ of unitary operators of the form $\I+T$, where $T$ is of Schatten class $2$, equipped with the Hilbert-Schmidt metric (Subs. \ref{uinfty2}), does not have property $(T)$ since it admits a continuous homomorphism to $U_C(\H)$ with a dense image. 
\end{example}

\section{$L^0(X,\mu;\T)$ does not have property $(T)$}

The construction in this section was inspired by the work of Solecki \cite{solecki}.

We will equip the groups of measurable maps with the $L^1$-metric, which is bi-invariant and induces the topology of convergence in measure.

Let $A$ be a subset of a standard Lebesgue measure space $(X,\mu)$ having $\mu$-measure half, and let $i\colon A\to X\setminus A$ be a measure preserving map. The map $i$ determines a continuous group homomorphism
$h=h_i\colon L^0(X,\mu;\T)\to L^0(A,\mu\vert_A;\T)$,
given by
\[h_i(f) = f\vert_A\cdot (f\vert_{X\setminus A}\circ i)^{-1}.\]
The kernel of this homomorphism consists of all ``$i$-periodic'' functions $f$ having the property $f\vert_A = f\vert_{X\setminus A}\circ i$.
The homomorphism $h$ is $1$-Lipschitz with regard to the metrics $L^1(\mu)$ and $L^1(\mu\vert_A)$. It is enough to verify the property near the identity function, $1$, because the distances are translation-invariant. We have
\begin{eqnarray*}
        \norm{h(f)-1}_{L^1(\mu\vert A)} &=& \int_A\left\bracevert (f\vert_A)\cdot (f\vert_{X\setminus A}\circ i)^{-1} - 1\right\bracevert\,d\mu\vert_A(x) \\
 &\leq&
 \int_A\left\bracevert 1 - f\vert_A + f\vert_A  - f\vert_A\cdot (f\vert_{X\setminus A}\circ i)^{-1} \right\bracevert\,d\mu\vert_A(x) \\
&\leq& \int_A\left\bracevert 1 - f\vert_A\right\bracevert \,d\mu\vert_A(x)+
\int_{X\setminus A}\left\bracevert 1 - f\vert_{X\setminus A}^{-1}\right\bracevert \,d\mu\vert_{X\setminus A}(x)\\
&=& \int_X\left\bracevert 1 - f(x)\right\bracevert\,d\mu(x) \\
&=&\norm {f-1}_{L^1(\mu)}.
\end{eqnarray*}

Denote by $\rho=\rho_A$ the unitary representation of $L^0(A,\mu\vert_A;\T)$ in $L^2(A,\mu\vert_A)$ by multiplication operators:
\begin{equation}
\label{eq:mult}
L^0(A,\mu\vert_A;\T)\ni f\mapsto [\rho_A(f)\colon g\mapsto fg]\in U(L^2(A,\mu\vert_A)).\end{equation}
(Here we interpret $\T=\{z\in\C\colon \abs z=1\}$.)
Define a unitary representation $\pi=\pi_i$ of the group $L^0(X,\mu;\T)$ in $L^2(A,\mu\vert_A)$ by
\[\pi_i = \rho_A\circ h_i.\]
Notice that for all $f\in L^0(X,\mu;\T)$,
\begin{equation}
\label{eq:prop}
\norm{\pi(f)(1_A)-1_A}_{L^2(\mu\vert_A)} \leq 
\norm{\pi(f)(1_A)-1_A}_{L^1(\mu\vert_A)}=\norm{h(f)-1_A}_{L^1(\mu\vert_A)} \leq \norm{f-1}_{L^1(\mu)}.\end{equation}
Here $1_A$ is the indicator function of $A$, a vector of norm $\sqrt 2/2$ in $L^2(A,\mu\vert_A)$.

From now on, we will identify $(X,\mu)$ with the unit interval $[0,1]$ equipped with the Lebesgue measure. Define
\[A_n =\bigcup_{i=0}^{n-1}\left[\frac {i} {n}, \frac{2i+1}{2n}\right)\]
and set
\[i_n(x) =  x+\frac 1{2n},
\]
a measure space isomorphism of $A_n$ with its complement.
Let $\pi$ be a unitary representation of $L^1(X,\mu;\T)$ in the space $\H=\oplus_{n\in\N_+}L^2(A_n,\mu\vert_{A_n})$,
which is the orthogonal sum of the countably many representations $\pi_{i_n}$, $n=1,2,3,\ldots$.
Notice that $\pi$ does not have an invariant unit vector, because none of the subrepresentations $\pi_{i_n}$ do.
We will show that $\pi$ has almost invariant vectors.

Let $K$ be a compact subset of $L^0(X,\mu;\T)$, and let $\e>0$. Approximate $K$ to within $\sqrt 2 \e/2$ with regard to the $L^1(\mu)$-distance with a finite set $F$ of simple functions taking constant values on each of the intervals
\[\left[\frac i n, \frac{i+1}n\right),~~i=0,1,\ldots,n-1,\]
where $n$ is sufficiently large. 

The functions in $F$ are invariant under the transformation $i_n$, where $n$ is chosen as above. Therefore,
the representation $\pi=\pi_{i_n}$ sends the set $F$ to the identity operator. 
For every function $f\in K$ at a $L^1(\mu)$-distance $<\sqrt 2\e/2$ from some $g\in F$, one has 
\[\norm{\pi(f)(1_{A_n})-1_{A_n}}_{\H}
=\norm{\pi(f)(1_{A_n})-1_{A_n}}_{L^2(\mu\vert_{A_n})}\leq \norm{f-g}_{L^1(\mu)}<\sqrt 2\e/2,\]
and since $\norm{1_{A_n}}_{\H}=\sqrt 2/2$, the vector $1_{A_n}$ is $(K,\e)$-almost invariant. \qed

\begin{remark}
        It is quite easy to see that the group $L^0(X,\mu;\T)$ has property $(OB)$, because it equals a finite power of every neighbourhood of zero. Thus, it provides the second example of a Polish group with property $(OB)$ and without property $(T)$. We do not know if the above construction can be extended to every compact group $K$ in place of $\T$.
\end{remark}

\section{Minimally almost periodic topological group with property $(T)$ without a finite Kazhdan set}

\subsection{Infinite products of Kazhdan sets in topological groups}

\begin{lemma}
        Let $I$ be an index set, and for each $i\in I$, let $Q_i$ be a Kazhdan set in a topological (possibly discrete) group $\Gamma_i$, with the same Kazhdan constant $\e>0$. Let $\Gamma$ be a subgroup of 
        the direct product $\prod_{i\in I}\Gamma_i$ containing both the direct sum $\oplus_{i\in I}\Gamma_i$ and the set $Q=\prod_{i\in I} Q_i$. Suppose $\Gamma$ is equipped with a group topology in which the direct sum  $\oplus_{i\in I}\Gamma_i$ is dense. 
Then $Q=\prod_{i\in I} Q_i$ is a Kazhdan set for the topological group $\Gamma$, with regard to any Kazhdan constant $\delta$ such that $0<\delta<\e$. 
\label{l:infq}
\end{lemma}

\begin{proof}  
        Let $\H$ be a unitary $\Gamma$-module, that is, a Hilbert space equipped with a strongly continuous unitary representation of $\Gamma$. For each subset $J\subseteq I$, denote $\H_J$ the Hilbert subspace consisting of all vectors invariant under each group $\Gamma_j$, $j\in J$. For all $i\in I\setminus J$, $j\in J$, $g\in\Gamma_i$, $h\in\Gamma_j$, and $\xi\in\H_J$, one has $h(g\xi)=gh\xi=g\xi$, and so $\H_J$ is a unitary $\oplus_{i\in I}\Gamma_i$-submodule, hence a $\Gamma$-submodule. 
        
        Identify $I$ with the smallest ordinal $\tau$ of cardinality $\abs{I}$. In particular, $\H_{\tau}=\H_I$ is the module of all $\Gamma$-invariant vectors. For each $0\leq\beta<\tau$ denote $\check\H_\beta$ the orthogonal complement of the unitary module $\H_{[0,\beta]}$ in $\H_{[0,\beta)}$. (Notice that $\H_{\emptyset}=\H$.) The orthogonal projections $\check\pi_{\beta}$ of $\H$ into $\check\H_\beta$, $\beta\leq\tau$, are $\Gamma$-equivariant, as is the orthogonal projection $\pi_{\tau}$ of $\H$ onto $\H_{\tau}$. The unitary module $\H$ decomposes into a direct sum of submodules $\check\H_\beta$, $\beta<\tau$, and $\H_{\tau}$.
        
        Now let $\xi$ be a $(Q,\delta)$-almost invariant vector in $\H$ of unit norm, for some $\delta<\e$. Set $\xi_{\beta}=\check\pi_{\beta}(\xi)$, $\xi_{\tau}=\pi_{\tau}(\xi)$, and $\delta_\beta=\norm{\xi_\beta}$. One has $\delta_{\beta}>0$ 
        for at most a countable set of ordinals $\beta$, and $\xi=\sum_{\beta\leq\tau}\xi_{\beta}$. It remains to show that $\norm{\xi_{\tau}}>0$, that is, $\norm{\sum_{\beta<\tau}\xi_{\beta}}<1$.
        
        Let $\beta<\tau$. If $\xi_\beta=0$, set $g_\beta=e\in\Gamma_\beta$. Else, 
        the action of $\Gamma_{\beta}$ on $\check\H_\beta$ has no invariant vectors, and by assumption there is $g_\beta\in Q_\beta$ satisfying $\norm{\xi_\beta-g_\beta\xi_\beta}\geq\delta_\beta\e$. Set $g=(g_\beta)_{\beta<\tau}\in Q$. Since the vectors $\xi_\beta-g_\beta\xi_\beta$ are pairwise orthogonal, one has
        \[\delta>\norm{g\xi-\xi} = \norm{g\sum_\beta\xi_\beta-\sum_\beta\xi_\beta}\geq \e\sqrt{\sum_\beta \delta_\beta^2}.
        \]
        This implies:
        \[\norm{\sum_\beta\xi_\beta}=\sqrt{\sum_\beta \delta_\beta^2}<\delta/\e<1,\]
        as required.
\end{proof}

\begin{remark}
        In a non-discrete topological group, a subgroup generated by a Kazhdan subset need not be dense, hence the assumptions of the lemma.
        Also, as is seen from the proof, if $I$ is uncountable, then already a subset of $Q$ consisting of all elements with countable supports is a Kazhdan set.
\end{remark}

\subsection{Groups of bounded maps}
For a group $\Gamma$ and an index set $I$, denote $b(I,\Gamma)$ the set of all functions from $I$ to $\Gamma$ each having a finite range. This is a subgroup of $\Gamma^I$, containing the direct sum of $I$ copies of $\Gamma$. Informally, this is a ``little $\ell^\infty$ with values in $\Gamma$.''

We will define a group topology on $b(I,\Gamma)$ as follows. For every finite subset $Q\subseteq\Gamma$, equip the product $Q^I$ (which is contained in our group) with the compact product topology. The group $b(I,\Gamma)$ is covered by a countable directed family of compact sets $Q^I$. Define a chain topology with regard to this cover: a subset $F$ of the group is closed if and only if every intersection $F\cap Q^I$ is closed in $Q^I$ for each finite $Q\subseteq\Gamma$.

To verify that the multiplication map $m$ is continuous, in accordance with the definition of the topology, it is enough to check that the restriction of $m$ to each $Q_1^I\times Q_2^I$ is continuous. The image of this set is contained in $m(Q_1\times Q_2)^I$. Now it is enough to observe that the map $m$ is continuous with regard to the product topology on $b(I,\Gamma)$ induced from $\Gamma^I$ (the Tychonoff product of $I$ copies of a countable discrete space), and this product topology induces the same compact topology on each $Q^I$. The same works for the inversion map. 

Finally, since our group topology is finer than the product topology, it is also  Hausdorff.

Such topologies are called $k_{\omega}$-topologies \cite{FST}, or else the weak topologies with regard to a countable cover by compact subsets, or the chain topologies. For instance, this is a standard topology on $CW$ complexes. Some of the $k_{\omega}$ spaces are metrizable ($\R^n$), others are not (the $LF$ space $\R^\infty$). The topological group $U(\infty)=\cup_nU(n)$, the union of an increasing chain of the finite rank unitary groups with the chain topology, is a $k_\omega$ group of some importance in representation theory. 

If a topological group $G$ is a $k_{\omega}$-space, that is, its topology is a $k_\omega$-topology with regard to a countable sequence of compact sets $K_n$, then $G$ is necessarily complete, already with regard to the left uniformity. We can assume that $K_n$ is an increasing sequence.
Given a (decreasing) sequence $(V_n)$ of neighbourhoods of the identity, one can select recursively a neighbourhood $V$ with the property that for each $m$ there is $n$ with $VK_m\subseteq V_nK_n$. If now $\mathcal F$ is a Cauchy filter, it follows that for each sequence $(V_n)$ there is $n$ with $V_nK_n\in {\mathcal F}$. By the diagonal argument, $\mathcal F$ intersects each neighbourhood of some $K_n$, and converges to an element of $K_n$. See \cite{graev_free} for this argument in the context of free topological groups where it probably appeared for the first time.

If the index set $I$ is countable, the group $b(I,\Gamma)$ is separable, as a countable union of metrizable compacta (Cantor sets), $Q^I$. In particular, for $\Gamma$ non-trivial, this topology is never discrete.

\begin{proposition}
If $\Gamma$ is infinite, then the topological group $b(I,\Gamma)$ is non-metrizable. 
\end{proposition}

\begin{proof}
        Otherwise, being metrizable and complete, it would satisfy the Baire category theorem, and one of the $Q^I$ would have a non-empty interior, so $Q^I\cdot Q^I=(Q\cdot Q)^I$ would have been open by a standard argument. But this set is not even relatively open in $Q_1^{I}$, where the finite set $Q_1$ is obtained from $Q\cdot Q$ by adding one extra point. 
\end{proof}

\begin{lemma}
        If $\Gamma$ is a minimally almost periodic discrete group, then $b(I,\Gamma)$ is a minimally almost periodic topological group.
        \label{l:map}
\end{lemma}

\begin{proof}
        Every continuous finite-dimensional unitary representation  $\pi\colon b(I,\Gamma)\to U(k)$ satisfies $\pi(\Gamma_n)=\{e\}$ for all $n$, hence the image of the direct sum group $\oplus_{i\in I}\Gamma_i$ is trivial. But this group is everywhere dense in $b(I,\Gamma)$ with regard to our $k_{\omega}$-topology (as it is dense in every set $Q^I$).
\end{proof}

\subsection{Kazhdan sets in groups of bounded maps}

Lemma \ref{l:infq} implies:

\begin{corollary}
        If $\Gamma$ is a discrete Kazhdan group, then $b(I,\Gamma)$ with the  $k_{\omega}$-topology is a Kazhdan topological group with a compact Kazhdan set. \qed 
        \label{l:q}
\end{corollary}

\begin{lemma}
        If $I$ is infinite and $\Gamma$ is nontrivial, $b(I,\Gamma)$ with the  $k_{\omega}$-topology does not admit a finite Kazhdan set (as a topological group).
        \label{l:nofinite}
\end{lemma}

\begin{proof} Let $F$ be a finite subset of $b(I,\Gamma)$. By the pigeonhole principle  and because of the boundedness of each $g\in F$, for some distinct $i,j\in I$ one will have $g_i=g_j$ for every $g\in F$. Denote $\pi=\pi_{\{i,j\}}$ the coordinate projection of $b(I,\Gamma)$ to $\Gamma_i\times\Gamma_j$; clearly, $\pi$ is continuous when the latter group is discrete. 
        The image of $F$ under $\pi$ is contained in the diagonal, $\Delta$, which is a proper subgroup of $\Gamma^2$. The rest is standard. The composition of $\pi$ with the left quasiregular representation of $\Gamma^2$ on $\ell^2(\Gamma^2/\Delta)$ has a vector fixed by $F$, but no globally fixed vector.
\end{proof}

\begin{remark}
        Since the group $b(I,\Gamma)$ is uncountable, of course it does not admit a finite Kazhdan set as a discrete group. But apriori it may admit one as a topological group, in the same way as $U(\ell^2)_s$ admits a Kazhdan set with two elements \cite{Bek3}. 
\end{remark}

\begin{example}
        Let $\Gamma$ be an infinite simple Kazhdan group (cf. \cite{ozawa} and references). Then the complete, separable topological group $b(\N,\Gamma)$ has property $(T)$ with a compact Kazhdan set by Corollary \ref{l:q}, yet admits no finite Kazhdan set by Lemma \ref{l:nofinite}. At the same time, it is minimally almost periodic by Lemma \ref{l:map}. 
        
        This gives a counter-example to the conjecture of Bekka. It would be interesting to find an example of a Polish group with the same combination of properties. For instance, the Polish group $\Gamma^\N$ (with the product topology) is Kazhdan with the same Kazhdan set $Q$, minimally almost periodic, but this author does not know whether $\Gamma$ can be so chosen that $\Gamma^\N$ has no finite Kazhdan set.
\end{example}

\subsection{Infinite products of Kazhdan groups} It is well-known that the class of groups with property $(T)$ is closed under finite products.
As another consequence of Lemma \ref{l:infq}, we have the following curious result.

\begin{theorem}
        The class of topological groups with property $(T)$ is closed under arbitrary (infinite) products, with the usual product topology.
\end{theorem}

This will follow from Lemma \ref{l:infq} once we assure that, given a family $\Gamma_i$, $i\in I$ of Kazhdan topological groups, we can always select  compact Kazhdan sets $Q_i\subseteq\Gamma_i$ with a prescribed Kazhdan constant (for instance, $1$).

\begin{lemma}
        Let $\Gamma$ be a topological group with property $(T)$. Let $(Q,\e)$ be a Kazhdan pair. Let $\delta>0$. Then for a suitably large natural $n$, the pair $((Q\cup Q^{-1})^n, \sqrt 2 -\delta)$ is a Kazhdan pair in the topological group $\Gamma$.
\end{lemma}

\begin{proof} 
        Assume the contrary, that is, for every $n$, there is a strongly continuous unitary representation $\pi_n$ of $\Gamma$ in a Hilbert space $\H_n$ which contains a $((Q\cup Q^{-1})^n, \sqrt 2 -\delta)$-almost invariant vector $\xi_n$ of unit norm, and yet has no invariant vectors, in particular, no vector is $(Q,\e)$-almost invariant. Then $\Gamma$ is unitarily represented in the metric ultraproduct $\H$ of the Hilbert spaces $\H_n$ with regard to some non-principal ultrafilter, $\mathcal U$, on the natural numbers (for this construction, see e.g. \cite{CCS}). The ultraproduct representation, $\pi$, is in general no longer strongly continuous, but this is unimportant. Denote $G$ a subgroup of $\Gamma$ algebraically generated by $Q$; note that for non-discrete topological groups, one generally does not expect $G$ to be dense in $\Gamma$.
        One sees easily that $\pi$ still has no $(Q,\e)$-almost invariant vectors, in particular no $G$-invariant vectors since $Q\subseteq G$. At the same time, there is a vector of unit norm which is $(G,\sqrt 2 -\delta)$-almost invariant, namely the equivalence class $[(\xi_n)]_{\mathcal U}$ of the sequence $(\xi_n)$ modulo the ultrafilter. This contradicts Proposition 1.1.5 in \cite{BdlHV}: for any group $G$ and each $\delta>0$, $(G,\sqrt 2 -\delta)$ is a Kazhdan pair. 
\end{proof}

\section{Some question marks}

We still do not know whether every amenable SIN group with property $(T)$ (that is, having a compact Kazhdan set), whose unitary representations separate points and closed subsets, is necessarily precompact. This would strengthen Theorem\ \ref{th:major}. 
Other questions are summarized in the table below, where f.p.c. stands for the fixed point on compacta property (extreme amenability), and m.a.p., for minimal almost periodicity.
  \vskip .3cm

\begin{tabular}{|l|c|c|c|c|c|c|c|}
\hline
Topological group & amenable & f.p.c. & SIN & m.a.p. & OB &  strong $(T)$ & $(T)$ \\
\hline\hline
$S_{\infty}$ & $\checkmark$ & $\times$ & $\times$ & $\checkmark$& $\checkmark$& $\checkmark$ & $\checkmark$ \\ \hline
$\Aut(\Q,\leq)$ & $\checkmark$ & $\checkmark$ & $\times$ & $\checkmark$& $\checkmark$& $\checkmark$ & $\checkmark$ \\ \hline
$U(n)$ & $\checkmark$ & $\times$ & $\checkmark$ & $\times$ & $\checkmark$ & $\checkmark$ iff $n\geq 6$ & $\checkmark$ \\ \hline
$U(\ell^2)_s$ & $\checkmark$ & $\checkmark$ & $\times$ & $\checkmark$& $\checkmark$& $\checkmark$ & $\checkmark$ \\ \hline
$U_C(\ell^2)$ & $\checkmark$ & $\checkmark$ & $\checkmark$ & $\checkmark$ & $\checkmark$ & $\times$ & $\times$ 
\\ \hline
$U(\infty)_2$ & $\checkmark$ & $\checkmark$ & $\checkmark$ & $\checkmark$ & $\times$ & $\times$ & $\times $ \\
\hline
$U(\ell^2)_u$ & $\times$ & $\times$ & $\checkmark$ & $\checkmark$ & $\checkmark$ & ? & ? \\ \hline
$U(R)_2$ & $\checkmark$ & $\checkmark$ & $\checkmark$ & $\checkmark$ &
$\checkmark$ & $\times$ & ? \\ \hline
$U(R^\omega)_2$ & $\times$ & $\times$ & $\checkmark$ & $\checkmark$ &
$\checkmark$ & ? & ? \\ \hline
$\Aut(X,\mu)_w$ & $\checkmark$ & $\checkmark$ & $\times$ & $\checkmark$& $\checkmark$& ? & ? \\ \hline
$\Aut(X,\mu)_u$ & $\times$ & $\times$ & $\checkmark$ & $\checkmark$& $\checkmark$& ? & ? \\ \hline
$\Aut^\ast(X,\mu)_w$ & $\checkmark$ & $\checkmark$ & $\times$ & $\checkmark$& $\checkmark$& ? & ? \\ \hline
$\Aut^\ast(X,\mu)_u$ & ? & ? & $\times$ & $\checkmark$& $\checkmark$& ? & ? \\ \hline
$L^0(X,\mu;K)$, $\mu$ diffuse & $\checkmark$ & $\checkmark$ & $\checkmark$ & $\checkmark$ &
$\checkmark$ & $\times$ & ? \\
same, if $K=\T$  & $\checkmark$ & $\checkmark$ & $\checkmark$ & $\checkmark$ &
$\checkmark$ & $\times$ & $\times$ \\
 \hline
$[{\mathscr R}]$, 
$\mathscr R$ hyperfinite:  &&&&&& &  \\
- measure-preserving&$\checkmark$ &$\checkmark$&$\checkmark$&$\checkmark$&$\checkmark$& $\times$ & ? \\
- non-singular &$\checkmark$ &$\checkmark$&$\times$&$\checkmark$& $\checkmark$ & ? & ? \\
 \hline
 $b(\N,\Gamma)$, $\Gamma$ inf. simple $(T)$ & $\times$ & $\times$ & ? & $\times$ &$\times$ &$\times$ &$\checkmark$ \\
\hline 
$\Gamma^{\N}$, $\Gamma$ inf. simple $(T)$ & $\times$ & $\times$ & $\checkmark$ & $\times$ &$\times$ & ? &$\checkmark$ \\
\hline 
\end{tabular}
\vskip .4cm

The author thanks Philip Dowerk and Andreas Thom, and Sven Raum for a useful remark (in Subs. \ref{ur2}), and the anonymous referee for a number of comments.

\end{document}